\documentclass[10pt,draft,reqno]{amsart}
      \makeatletter
     \def\section{\@startsection{section}{1}%
     \z@{.7\linespacing\@plus\linespacing}{.5\linespacing}%
     {\bfseries%\normalfont\scshape
     \centering
     }}
     \def\@secnumfont{\bfseries}
     \makeatother
\newtheorem{theorem}{Theorem}[section]
\newtheorem{lemma}[theorem]{Lemma}

\newtheorem{corollary}[theorem]{Corollary}
\theoremstyle{definition}

\theoremstyle{remark}
\newtheorem{remark}[theorem]{Remark}
\numberwithin{equation}{section}
\def \a{{\alpha}}
\def \b{{\beta}}
\def \D{{\Delta}}
\def \DD{{\mathcal D}}
\def \d{{\delta}}
\def \e{{\varepsilon}}

\def \g{{\gamma}}
\def \G{{\Gamma}}
\def \k{{\kappa}}
\def \l{{\lambda}}
\def \mh{{\hat {\mu} }}
\def \n{{\eta}}
\def \o{{\omega}}
\def \O{{\Omega}}
\def \p{{\varphi}}
\def \t{{\vartheta}}
\def \ttau{{\theta}}
\def \m{{\mu}}
\def \s{{\sigma}}
\def \Y{{\Theta}}

\def \aa{{d}}
\def \A{{\mathcal A}}
\def \bp{{{\mathcal B}_p}}
\def \C{{\mathcal C}}
\def \cc{{{\underline{c}}}}
\def \E{{\bf E}\, }
 \def \N{{\bf N}}
 \def \NN{{\bf N}}
 \def \pp{ {{\bar{p}}}}
\def \P{{\bf P}}
\def \q{{\quad}}
\def \Q{{\bf C}}
\def \qq{{\qquad}}
\def \R{{\bf R}}
\def \ra{{\rightarrow}}
\def \rr{ {{\bf r}}}
\def \T{{\bf T}}
\def \td{{\widetilde D}}
\def \ua{{\underline{a}}}
\def \uc{{\underline{c}}}
\def \ud{{\underline{d}}}
\def \uu{ {{\underline{\t}}}}

\def \X{{\cal X}}
\def \dd{{\rm }}
 \def \ut{{\underline{\tau}}}

\def \noi{{\noindent}}

\def\E{{\mathbb E}}

\def \T{{\mathbb T}}

\def\P{{\mathbb P}}

\def\R{{\mathbb R}}

\def\Z{{\mathbb Z}}
 
\def\Q{{\mathbb Q}}

\def\N{{\mathbb N}}

\def\NS{{\mathbb N}^*}

\def\C{{\mathbb C}}  
  
\font\phh=cmcsc10
at  8 pt
\scrollmode

\font\gsec= cmb10 at 10 pt
  \font\ggum= cmb10 at 10,8  pt
\font\phh=cmr10 at  8,2 pt
\font\ph=cmcsc10 at 10 pt  \scrollmode
\font\gsec= cmb10 at 9,5  pt
\font\gssec= cmb10 at 8,2  pt

 \title[A sharp correlation inequality]{A sharp correlation inequality   with an application to almost sure   local
 limit theorem} 
 
  \author{ Michel Weber 
}
\address{ Michel Weber: IRMA, Universit\'e
Louis-Pasteur et C.N.R.S.,   7  rue Ren\'e Descartes, 67084
Strasbourg Cedex, France. }
\email{michel.weber@math.unistra.fr
}
 \urladdr{http://www-irma.u-strasbg.fr/$\sim$weber/}
\begin{document}
 \maketitle
   
  %%%%%%%%%%%%%%%%%%%%%%%%%%%%%%%%%%%%%%%%%%%%%%%%%%%%%%%%%%%%%%%%%%%%%%%%%%%%%%%%%%%%%%%%%%%%%%
\vskip -15pt
 \centerline{---------------------------------------------------------------------------------------------------------}
\vskip - 25pt
 \begin{abstract}\noi   We   prove   a
   new  sharp correlation inequality for sums of i.i.d. square integrable lattice distributed random variables. We also apply it to
establish  an almost sure local limit theorem for iid   square integrable random variables taking
values in an arbitrary lattice. This extends a recent similar result jointly obtained with
Antonini-Giuliano, under a slightly stronger absolute moment assumption (of order $2+u$ with $u>0$). The approach used to treat the  case
$u>0$ breaks down when
$u=0$.   MacDonald's concept of the Bernoulli
part of a random variable is used in a crucial way to remedy this.
     \end{abstract}
   \vskip 4pt  
\noi {\gssec 2010 AMS Mathematical Subject Classification}: {\phh Primary: 60F15, 60G50 ;
Secondary: 60F05}.  \par\noi  
{{\gssec Keywords and phrases}: {\phh Correlation inequality,  i.i.d. random variables,
lattice distributed,  Bernoulli part, square integrable, local limit theorem, almost sure version}. } 
 %%%%%%%%%%%%%%%%%%%%%%%%%%% %%%%%%%%%%%%%%%%%%%%%%%%%%% %%%%%%%%%%%%%%%%%%%%%%%%%%% %%%%%%%%%%%%%%%%%%%%%%%%%%% %%%%%%%%%%%%%%%%%%%%%%%%%%%
\vskip  3pt
 \centerline{------------------------------------------------------------------------------------------------------------}
\vskip   5pt
 \noi \centerline{This is the extended version of a paper that is to appear in}
 \vskip 1pt
\centerline{{\it Probability and Mathematical Statistics}.} 
   
\section{Introduction} 
 Throughout this work, we are concerned with   i.i.d. square integrable random
variables having lattice distribution. Let  $v_0 $ and $D>0$ be some reals and let   $\mathcal L(v_0,D)$ be the lattice defined by the
sequence $v_k=v_0+Dk$, $k\in \Z$. Consider
  a random variable $X$ such that  $\P\{X
\in\mathcal L(v_0,D)\}=1$. We assume that $D$ (the span of $X$) is maximal, i.e. there is no integer multiple $D'$ of $D$ for which $\P\{X
\in\mathcal L(v_0,D')\}=1$. 
  We further assume
  \begin{equation}\label{moment}\hbox{$ \E X$ \ and \  $ \E X^2$ \quad are   finite}  . 
 \end{equation}
Let $\m =\E X$ and $\s^2=\E X^2-(\E X)^2$, which we assume to be positive (otherwise $X$ is degenerated). Under these assumptions, the local limit theorem
holds.    Let  $ \{X_k, k\ge 1\}$ be independent copies of
$X$, and consider their partial sums
$S_n=X_1+\ldots +X_n$, $n\ge 1$. To be precise, we have   (\cite{G}, \S43), 
   \begin{equation}\label{llt}  \lim_{n\to \infty}\sup_{N=v_0n+Dk }\Big|  \sqrt n \P\{S_n=N\}-{D\over  \sqrt{ 2\pi}\s}e^{- 
{(N-n\m )^2\over  2 n \s^2} }\Big|= 0. 
\end{equation} 

 Now let   $\k_n\in \mathcal L(nv_0,D)$, $n=1,2,\ldots$ be a   sequence of reals satisfying
\begin{equation}\label{2}  \lim_{n\to \infty} { \k_n-n\m   \over    \sqrt{  n}  }= \k .
\end{equation}  
  The central result of the paper is the following correlation inequality which  we believe to be  hardly improvable. 
   \begin{theorem} \label{t1} Assume that
  \begin{equation}\label{basber}\hbox{ $ \P\{X=k\}\wedge\P\{X=k+1\}\,>0$  \qquad  for some   $k\in\Z$.}   
 \end{equation}  Then there  exists a constant
$C $  depending on the sequence $\{\k_n, n\ge 1\}$ such that for all $1\le m<n$,
\begin{eqnarray*} & &\sqrt{nm} \, \Big|\P\{S_n=\k_n, S_m=\k_m\}-\P\{S_n=\k_n \}\P\{  S_m=\k_m\} \Big|  
\cr  &\le &  C \Big\{     {   1\over 
 \sqrt{n\over m}-1   } +      {n^{1/2}  \over
 (n-m) ^{3/2}} \Big\} .
\end{eqnarray*}   
 \end{theorem}  

 \begin{corollary} \label{cort1}      Let $0<c<1$. Under assumption (\ref{basber}),  there   exists a constant
$C_{  c} $  such that for all $1\le m\le cn$,
\begin{eqnarray*} \sqrt{nm} \,\Big|\P\{S_n=\k_n, S_m=\k_m\}-\P\{S_n=\k_n \}\P\{  S_m=\k_m\} \Big|  
 &\le &  C_{c} \, \sqrt{{m\over n}}. \end{eqnarray*}
 \end{corollary} 
\begin{remark}\label{bbl}\rm    
Condition (\ref{basber}) seems to be somehow artificial. It is for instance clearly not   satisfied if  $\P\{ X\in
{\mathcal N}\} =1$   where
${\mathcal N}=\{\nu_j, j\ge 1\}$ is an increasing sequence of integers such that $\nu_{j+1}-\nu_j >1$ for all $j$.
  This already defines  a large class of examples. However, condition (\ref{basber}) is   natural in our setting. By the local limit
theorem (\ref{llt}), under condition (\ref{2}),
   $$ \lim_{n\to \infty}\sqrt n \P\{S_n=\ell_n\}= {D\over  \sqrt{ 2\pi}\s}e^{- 
{\k^2\over  2  \s^2} }, \qq (\ell_n\equiv \k_n\ {\rm or}\ \ell_n\equiv\k_n+1)  . 
$$
Then for some $n_\k<\infty$, $\P\{S_n=\k_n\}\wedge \P\{S_n=\k_n+1\}>0 $ if $n\ge n_\k$. Changing $X$ for   $X'=S_{n_\k}$, we see that $X'$ satisfies
(\ref{basber}). 
\end{remark}  

When $X$ has stronger integrability property, to be precise when $\E |X|^{2+\e}<\infty$   
for some positive $\e$,   we proved in \cite{GW3},
(Proposition 6) a
similar result:  \begin{eqnarray}\label{GW} & &\sqrt{nm} \,\Big|\P\{S_n=\k_n, S_m=\k_m\}-\P\{S_n=\k_n \}\P\{  S_m=\k_m\}
\Big| 
\cr 
 &\le &   C \Big({{1} \over {\sqrt{{n\over m}}-1}}+ \sqrt {n\over n-m}\ {{1} \over {(n-m)^\alpha}}\Big), \end{eqnarray}
with (here and below) $\a=\e/2$. Condition (\ref{basber}) was not needed.
 The second inequality of Theorem \ref{t1}  follows in that case directly from
(\ref{GW}). The proof uses crucially a local limit theorem with remainder term
 \begin{theorem}   {\rm (\cite{IL} Theorem 4.5.3)}    {\it  Let $F$ denote the distribution function of $X$. In order that  the property
\begin{equation}\label{r}   \sup_{N=an+dk}\Big|  \sqrt n \P\{S_n=N\}-{d\over  \sqrt{ 2\pi}\s}e^{- 
{(N-n\m )^2\over  2 n \s^2} }\Big| ={\mathcal O}\big(n^{-\alpha}\big) ,
\  0<\a<1/2 , 
 \end{equation}  holds, 
 it is necessary and sufficient that the following conditions be satisfied}:
 \begin{eqnarray*} 1) & & d=D,\cr 
2) & & \hbox{as $u\to \infty$,} \qq\int_{|x|\ge u} x^2 F(dx) = \mathcal O(u^{-2\a}) 
\end {eqnarray*}
\end{theorem}
 When $\e=0$,   this   can obviously no longer be applied, and another approach has to be implemented. Notice that even
when
$\e>0$ our result is  stronger than (\ref{GW}).
\smallskip \par 
An application of Theorem \ref{t1} is given  in Section 4. We obtain an almost sure local limit theorem   for i.i.d. square
integrable   lattice distributed random variables.  
%%%%%%%%%%%%%%%%%%%%%%%%%%%%%%%%%%%%%%%%%%%%%%%%%%%%%%%%%%%%%%%%%%%%%%%%%%%%%%%%%%%%%%%%%%%%%%%%%%%%%%%%%%%%%%%%%%%%%%%%%%%%%%
%%%%%%%%%%%%%%%%%%%%%%%%%%%%%%%%%%%%%%%%%%%%%%%%%%%%%%%%%%%%%%%%%%%%%%%%%%%%%%%%%%%%%%%%%%%%%%%%%%%%%%%%%%%%%%%%%%%%%%%%%%%%%%
%%%%%%%%%%%%%%%%%%%%%%%%%%%%%%%%%%%%%%%%%%%%%%%%%%%%%%%%%%%%%%%%%%%%%%%%%%%%%%%%%%%%%%%%%%%%%%%%%%%%%%%%%%%%%%%%%%%%%%%%%%%%%%
%%%%%%%%%%%%%%%%%%%%%%%%%%%%%%%%%%%%%%%%%%%%%%%%%%%%%%%%%%%%%%%%%%%%%%%%%%%%%%%%%%%%%%%%%%%%%%%%%%%%%%%%%%%%%%%%%%%%%%%%%%%%%%
%%%%%%%%%%%%%%%%%%%%%%%%%%%%%%%%%%%%%%%%%%%%%%%%%%%%%%%%%%%%%%%%%%%%%%%%%%%%%%%%%%%%%%%%%%%%%%%%%%%%%%%%%%%%%%%%%%%%%%%%%%%%%%
 
 \section{Preliminary Results} 
    Here we follow an important approach   due to MacDonald (\cite{M}, see also \cite{MD}). 
Let $0<\t< 1$ be fixed.   Put 
$$ f(k)= \P\{X= v_k\}, \qq k\in \Z .$$  We assume that there exists a sequence $\ut=\{ \tau_k, k\in \Z\}$ of   non-negative reals such that
$$  \tau_{k-1}+\tau_k\le 2f(k), \qq \forall k\in \Z,\qq\sum_{k\in \Z}  \tau_k =\t. $$ 
  If we choose $\t=\t_X=
\sum_{k\in \Z}  f(k)\wedge f(k+1)$, then this is realized with    $\tau_k= f(k)\wedge f(k+1)$. Notice that $\t_X<1$. Indeed,    
let  $k_0$ be some integer such that 
$f(k_0)>0$. Then
$$\sum_{k=k_0}^\infty  f(k)\wedge f(k+1)\le \sum_{k=k_0}^\infty  f(k+1)=\sum_{k=k_0+1}^\infty  f(k )
$$ 
And so  $ \t_X\le       \sum_{k< k_0 }  f(k) +\sum_{k=k_0+1}^\infty  f(k )<1$.     
\smallskip\par
Notice also that $\t_X>0 $. This follows from assumption
(\ref{basber}), and  is further necessary in order to make this approach efficient.
   MacDonald's construction applies to
the slightly more general  case we consider, and is even  easier to present.
 We define a pair of random variables $(V,\e)$   as follows.
For
$k\in\Z$, 
\begin{eqnarray}\label{ve}  \P\{ (V,\e)=( v_k,1)\}&=&\tau_k,      \cr &\cr
 \P\{ (V,\e)=( v_k,0)\}&=&f(k) -{\tau_{k-1}+\tau_k\over
2}    .  
\end{eqnarray}
This is well defined by assumption.    Observe that
$$\sum_{k\in\Z}  \Big[\P\{ (V,\e)=( v_k,1)\}+\P\{ (V,\e)=( v_k,0)\}\Big] =\sum_{k\in\Z}  f(k)+ {1\over
2}\sum_{k\in\Z}  \big[\tau_{k }-\tau_{k-1}\big]  =1. $$
     \begin{lemma}  We  have  for
$k\in\Z$
$$\P\{ V=v_k\} = f(k)+ {\tau_{k }-\tau_{k-1}\over 2} ,$$ 
and 
 $\P\{ \e=1\}=1-\P\{ \e=0\}  =\t$.  \end{lemma} 
\begin{proof} Plainly,
\begin{eqnarray*}\P\{ V=v_k\}& = & \P\{ (V,\e)=( v_k,1)\}+ \P\{ (V,\e)=( v_k,0)\}\cr
&=&f(k)+ {1\over 2}\big[\tau_{k }-\tau_{k-1}\big].
\end{eqnarray*} 
Further  $\P\{ \e=1\}  =\sum_{k\in\Z} \P\{ (V,\e)=( v_k,1)\}=\sum_{k\in\Z} \tau_{k }=\t  $. 
\end{proof} 
 \begin{lemma} Let $L$ be a Bernoulli random variable ($\P\{L=0\}=\P\{L=1\}=1/2$)  which is independent
from $(V,\e)$, and put
 $Z= V+ \e DL$. 
We have $Z\buildrel{\mathcal D}\over{ =}X$. 
\end{lemma}
\begin{proof} Indeed,
\begin{eqnarray*}\P\{Z=v_k\}&=&\P\big\{ V+\e DL=v_k, \e=1\}+ \P\big\{ V+\e DL=v_k, \e=0\} \cr 
&=&{\P\{ V=v_{k-1}, \e=1\}+\P\{
V=v_k, \e=1\}\over 2} +\P\{ V=v_k, \e=0\}
\cr&=& {\tau_{k-1}+ \tau_{k }\over 2} +f(k)-{\tau_{k-1}+ \tau_{k
}\over 2}   
\cr&=& f(k).
\end{eqnarray*}
\end{proof} 
    Now let $\{X_j, j\ge 1\}$ be independent copies of $X$. According to the previous construction, we
may  
 associate to them a  
sequence $\{(V_j,\e_j, L_j), j\ge 1\}$ of independent copies of  $(V ,\e , L )$  such that 
 $$ \{V_j+\e_jD L_j, j\ge 1\}\buildrel{\mathcal D}\over{ =}\{X_j, j\ge 1\}  . $$
Further $\{(V_j,\e_j), j\ge 1\}  
 $   and $\{L_j, j\ge 1\}$ are independent sequences. 
 And $\{L_j, j\ge 1\}$ is  a sequence  of independent Bernoulli random variables. Set 
\begin{equation}\label{dec} S_n =\sum_{j=1}^n X_j, \qq  W_n =\sum_{j=1}^n V_j,\qq M_n=\sum_{j=1}^n \e_jL_j,  \qq B_n=\sum_{j=1}^n \e_j .
\end{equation}

We   notice that $M_n$ is   a sum of exactly $B_n$ Bernoulli random variables. The Lemma below is now immediate.
 \begin{lemma} \label{lemd}We have the
representation  
$$ \{S_n, n\ge 1\}\buildrel{\mathcal D}\over{ =}  \{ W_n  + DM_n, n\ge 1\} .$$
And  $M_n\buildrel{\mathcal D}\over{ =}\sum_{j=1}^{B_n } L_j$. 
 \end{lemma} 

We need an extra lemma.
 %%%%%%%%%%%%%%%%%%%%%%%%%%%%%%%%%%%%%%%%%%%%%%%%lemma%%%%%%%%%%%%%%%%%%%%%%%%%%%%%%%%%%%%%%%%%%%%%%%%%%%%%%%%%%%%%%%%%%%%%%%%%%%%% We
 \begin{lemma} \label{lemk} Let $0<\theta\le \t$. For any positive integer $n$, we have $$\P\{B_n\le \theta n\}\le \Big( {1-\t\over
1-\theta}\Big)^{n(1-\theta)}\Big( {
\t\over
\theta}\Big)^{n \theta}.$$ Let $1-\t<\rho<1$. There exists  
$ 0< \theta <\t$,  $\theta=\theta(\rho,\t)$ such   that  for any positive integer $n$
 \begin{eqnarray*}  \P\{B_n\le \theta n\}\le \rho^n.
\end{eqnarray*} \end{lemma}
\begin{proof}   By Tchebycheff's inequality, for any $\l \ge 0$,
 \begin{eqnarray*}\P\{B_n\le \theta n\}&=&\P\{e^{-\l B_n} \ge e^{-\l\theta n}\}\le e^{-\l\theta n}\E e^{ \l B_n}=\big(e^{ \l\theta  }\E e^{
 \l
 \e}\big)^n \cr &=& \Big(e^{ \l\theta  }\big[1-\t(1-e^{-\l})\big] \Big)^n.
 \end{eqnarray*}
 Put  $x=e^\l  $, ($x\ge 1$) and let $\p(x)=x^{ \theta  }\big[1-\t(1-x^{-1})\big]$. Then $\P\{B_n\le \theta n\}\le \p(x)^n $. We have
 $\p'(x)=x^{\theta-2}(x\theta(1-\t)-(1-\theta)\t)$. Thus $\p$ reaches its minimum   at the value  $x_0={(1-\theta)\t\over \theta(1-\t)}.$
 And we have  $ \p(x_0)=\psi(\theta)$, where we put
 $$ \psi(\theta)=\Big( {1-\t\over 1-\theta}\Big)^{1-\theta}\Big( { \t\over \theta}\Big)^{ \theta}, \qq 0<\theta\le \t.$$
  We note that $\psi(\t )=1$, $\lim_{\theta\to 0+} \psi(\theta)= 1-\t $ and $\psi$  is nondecreasing ($(\log \psi)'(\theta)=\log\big( {
 \t\over 1-\t }\big/{ \theta \over 1-\theta }\big)\ge 0$, $0<\theta\le \t$). 
  Let $1-\t<\rho<1$. We
 may   select  
 $ 0< \theta_{\rho,\t}<\t$ depending on $\rho,\t$ only such   that $\psi(\theta)=\rho$. This yields
   the bound
 \begin{eqnarray}\P\{B_n\le \theta n\}\le \rho^n.
 \end{eqnarray}
 \end{proof}  
  
\medskip\par 
We choose 
$$\rho=  1-(\t/2) ,$$
 and let $0< \theta <\t$ such that in view of the preceding Lemma $\P\{B_n\le \theta n\}\le \rho^n$ and
$\P \{B_n-B_m\le  \theta (n-m)\}\le \rho^{n-m} $ for all integers $n>m\ge 1$. 
 %%%%%%%%%%%%%%%%%%%%%%%%%%%%%%%%%%%%%%%%%%%%%%%%%%%%%%%%%%%%%%%%%%%%%%%%%%%%%%%%%%%%%%%%%%%%%%%%%%%%%%%%%%%%%%%%%%%%%%%%%%%%%%
%%%%%%%%%%%%%%%%%%%%%%%%%%%%%%%%%%%%%%%%%%%%%%%%%%%%%%%%%%%%%%%%%%%%%%%%%%%%%%%%%%%%%%%%%%%%%%%%%%%%%%%%%%%%%%%%%%%%%%%%%%%%%%
%%%%%%%%%%%%%%%%%%%%%%%%%%%%%%%%%%%%%%%%%%%%%%%%%%%%%%%%%%%%%%%%%%%%%%%%%%%%%%%%%%%%%%%%%%%%%%%%%%%%%%%%%%%%%%%%%%%%%%%%%%%%%%
%%%%%%%%%%%%%%%%%%%%%%%%%%%%%%%%%%%%%%%%%%%%%%%%%%%%%%%%%%%%%%%%%%%%%%%%%%%%%%%%%%%%%%%%%%%%%%%%%%%%%%%%%%%%%%%%%%%%%%%%%%%%%%
%%%%%%%%%%%%%%%%%%%%%%%%%%%%%%%%%%%%%%%%%%%%%%%%%%%%%%%%%%%%%%%%%%%%%%%%%%%%%%%%%%%%%%%%%%%%%%%%%%%%%%%%%%%%%%%%%%%%%%%%%%%%%%
 \section{Proof of Theorem \ref{t1}}   
     Put  
\begin{equation}\label{Y}Y_n= \sqrt n \big({\bf 1}_{\{S_n=\kappa_n\}}-\P\{S_n=\kappa_n\} \big)  .
\end{equation}
 We have to establish that there  a constant
$C $  such that for all $1\le m<n$
\begin{eqnarray} \label{C1}\big|\E   Y_nY_m\big|  
 &\le &  C  \Big\{     {   1\over 
 \sqrt{n\over m}-1   } +      {n^{1/2} \over
 (n-m) ^{3/2}} \Big\} .
\end{eqnarray}   And   given  $0<c<1$, that there exists a constant
$C_{  c} $  such that for all $1\le m\le cn$,
\begin{eqnarray}\label{C2} \big|\E   Y_nY_m\big|  
 &\le &  C_{c} \, \sqrt{{m\over n}}. \end{eqnarray}
     We denote by $\E_{(V,\e)}$, $\P_{(V,\e)}$ (resp. $\E_{\!L}$, $\P_{\!L}$) the expectation and probability symbols
relatively to the
$\s$-algebra generated by the sequence   $\{(V_j,\e_j), j=1, \ldots, n\}$ (resp. $\{L_j , j=1, \ldots, n\}$). We know that these
 algebra are independent. 
Let $n>m\ge 1$.    Then
  \begin{eqnarray*}\label{cov}{\E Y_nY_m\over \sqrt{ n m}} &= & 
 \P
\{S_n =\kappa_n, S_m =\kappa_m \}-\P \{S_n =\kappa_n \}\P \{S_m =\kappa_m \}  \cr &=& \P \{S_m =\kappa_m\}\P\{ S_n-S_m
=\kappa_n -\kappa_m\}-\P
\{S_n =\kappa_n \}\P \{S_m =\kappa_m \} 
\cr &= &  \P \{S_m =\kappa_m\}\big(\P\{ S_{n-m} =\kappa_n -\kappa_m\}- \P \{S_n =\kappa_n \}\big),
\end{eqnarray*}
we get  for $n>m$
\begin{equation}\label{basic}\E Y_nY_m=\sqrt{   m}\P \{S_m =\kappa_m\}\,\sqrt{ n  }\big(\P\{ S_{n-m} =\kappa_n -\kappa_m\}- \P \{S_n
=\kappa_n
\}\big).
\end{equation}
 Further    when $n=m$, by (\ref{llt})
\begin{equation}\label{basic0}\E Y_n^2=n \P\{S_n=\kappa_n\}\big(1-\P\{S_n=\kappa_n\} \big) ={\mathcal O}(\sqrt n). 
\end{equation} 

 Now \begin{eqnarray*}
A &:=& \sqrt{ n  }\Big(\P\{ S_{n }-S_{ m} =\kappa_n -\kappa_m\}- \P \{S_n =\kappa_n \}\Big)
\cr 
& =& \sqrt{ n  }\E \big({\bf 1}_{\{B_n\le n\theta\}}+{\bf 1}_{\{B_n> n\theta\}}\big)\Big({\bf 1}_{\{ S_{n }-S_{ m} =\kappa_n -\kappa_m\}}-
{\bf 1}_{
\{S_n =\kappa_n \}}\Big)
\end{eqnarray*} 
By Lemma \ref{lemk}
\begin{eqnarray*} \sqrt{ n  }\, \E  {\bf 1}_{\{B_n\le n\theta\}} \big|{\bf 1}_{\{ S_{n }-S_{ m} =\kappa_n -\kappa_m\}}-
{\bf 1}_{
\{S_n =\kappa_n \}}\big|&\le& \sqrt {n}\, \rho^n. 
\end{eqnarray*}
Thus 
\begin{eqnarray} \label{e0}\Big|A-\sqrt{ n  }\,\E  {\bf 1}_{\{B_n> n\theta\}} \big({\bf 1}_{\{ S_{n }-S_{ m} =\kappa_n
-\kappa_m\}}- {\bf 1}_{
\{S_n =\kappa_n \}}\big)\Big|\le \sqrt n\, \rho^n.
\end{eqnarray} 
We can write in view of Lemma \ref{lemd}
\begin{eqnarray}\label{dep}& &\sqrt{ n  }\,\E  {\bf 1}_{\{B_n> n\theta\}} \Big({\bf 1}_{\{ S_{n }-S_{ m} =\kappa_n
-\kappa_m\}}- {\bf 1}_{
\{S_n =\kappa_n \}}\Big)\cr&= &\sqrt{ n 
}\E_{(V,\e)} {\bf 1}_{\{B_n> n\theta\}} \Big(\P_{\!L}
\Big\{D  \sum_{j=m+1}^n \e_jL_j =\kappa_n-\kappa_m-(W_{n  }-W_m)\Big\}
\cr & &\  -  \P_{\!L}
\big\{D \sum_{j= 1}^n \e_jL_j  =\kappa_n-W_n
\big\}\Big) 
\end{eqnarray}
Observe that if $B_n=B_m$, then $\sum_{j= 1}^n \e_jL_j =\sum_{j= 1}^m \e_jL_j $. Thus 
$$\Big\{D  \sum_{j=m+1}^n \e_jL_j =\kappa_n-\kappa_m-(W_{n  }-W_m)\Big\}=\Big\{W_{n 
}-W_m =\kappa_n-\kappa_m \Big\}.$$ 
So that (\ref{dep}) may be continued with 
\begin{eqnarray}\label{dep1}  &= & \sqrt{ n 
}\Big\{\E_{(V,\e)} {\bf 1}_{\{B_n> n\theta, B_n=B_m\}}\big({\bf 1}_{\{  W_{n  }-W_m=\kappa_n-\kappa_m   \} }
\cr & &\    - \P_{\!L}
\big\{D \sum_{j= 1}^n \e_jL_j  =\kappa_n-W_n
\big\}\big) \Big\}
\cr 
&  &+\sqrt{ n 
}\Big\{\E_{(V,\e)} {\bf 1}_{\{B_n> n\theta, B_n>B_m\}} \Big[\P_{\!L}
\Big\{D  \sum_{j=m+1}^n \e_jL_j =\kappa_n-\kappa_m-
\cr & &\ (W_{n  }-W_m)\Big\}
   -  \P_{\!L}
\big\{D \sum_{j= 1}^n \e_jL_j  =\kappa_n-W_n
\big\}\Big] \Big\}
\cr &:=& A'+ A''.
\end{eqnarray} 
    We bound $A'$ as follows
\begin{equation}\label{e1}|A'| \le \sqrt{ n 
}  \P 
\big\{  B_n=B_m\}  =\sqrt n\, 2^{-(n-m)}.
 \end{equation}
Now concerning $A''$, we have $$\sum_{j= 1}^n \e_jL_j\buildrel{\mathcal D}\over{ =}\sum_{j=1}^{B_n } L_j \qq \sum_{j=m+1}^n \e_jL_j
 \buildrel{\mathcal D}\over{ =}\sum_{j=B_m+1}^{B_n } L_j.$$
Now we need a local limit theorem for Bernoulli
sums.  By applying Theorem 13 in Chapter 7 of \cite{[P]}, we obtain   
 \begin{eqnarray}\label{lltber} \sup_{z}\, \Big|\sqrt N\, \P\big\{\sum_{j=1}^{N } L_j=z\} \big\} -{2\over \sqrt{2\pi}}e^{-{(z-(N/2))^2\over (N/2)}}\Big| =o\Big({1\over N}\Big). 
 \end{eqnarray} Therefore 
$$ \Big|   \P_{\!L} \big\{D\sum_{j=1}^{B_n } L_j =\kappa_n-W_n \big\} -{2e^{-{(\kappa_n-W_n-(B_n/2))^2\over D^2(B_n/2)}}\over
\sqrt{2\pi B_n}} \Big| =o\Big({1\over B_n^{3/2}}\Big). $$
And  on the set $\{ B_n>B_m\}$
$$ \Big|   \, \P_{\!L} \big\{D\sum_{j=1}^{B_{n }-B_{ m} } L_j =\kappa_n-\kappa_m-(W_{n  }-W_m) \big\}
\qq \qq\qq\qq$$ $$ \qq-{2e^{-{(\kappa_n-\kappa_m-(W_{n  }-W_m)-(B_{n }-B_{ m})/2 )^2\over D^2(B_{n }-B_{ m})/2 }}\over
\sqrt{2\pi (B_{n }-B_{ m})} } \Big| =o\Big({1\over
(B_{n }-B_{ m})^{3/2}}\Big). $$ 
%%%%%%%%%%%%%%%%%%%%%%%%%%%%%%%%%%%%%%%%%%%%%%%%%%%%%%%%%%%%%%%%%%%%%%%%%%%%%%%%%%%%%%%%%%%%%%%%%%%%%%%%%%%%%%%%%%%%%%%%%%%%%%
 It follows that 
 \begin{eqnarray}\label{e2} |A''| &\le & \sqrt{ n }\ \bigg|\E_{(V,\e)}{\bf 1}_{\{B_n> n\theta, B_n>B_m\}}
\Big\{{2e^{-{(\kappa_n-W_n-(B_n/2))^2\over D^2(B_n/2)}}\over
\sqrt{2\pi B_n}} 
\cr 
& &\quad- {2e^{-{(\kappa_n-\kappa_m-(W_{n  }-W_m)-(B_{n }-B_{ m})/2 )^2\over D^2(B_{n }-B_{ m})/2 }}\over
\sqrt{2\pi (B_{n }-B_{ m})} }\Big\}\bigg|
\cr 
& &\quad+C_0\sqrt{ n }\, \E_{(V,\e)}{\bf 1}_{\{B_n> n\theta, B_n>B_m\}}\Big\{ {1\over B_{n }^{3/2}}+ \Big({1\over
(B_{n }-B_{ m})^{3/2}}\Big)\Big\}\cr 
&:=& A''_1+C_0A''_2.
\end{eqnarray}
%%%%%%%%%%%%%%%%%%%%%%%%%%%%%%%%%%%%%%%%%%%%%%%%%%%%%%%%%%%%%%%%%%%%%%%%%%%%%%%%%%%%%%%%%%%%%%%%%%%%%%%%%%%%%%%%%%%%%%%%%%%%%%
And the constant $C_0$ comes from the Landau symbol $o$ in (\ref{lltber}). The second term is easily estimated. Indeed, 
\begin{eqnarray}\label{e6}A''_2&= &\sqrt n \,\E_{(V,\e)}{\bf 1}_{\{B_n> n\theta, B_n>B_m\}}
 \Big( {1\over B_{n }^{3/2}}+{1\over (B_{n }-B_{ m})^{3/2}}\Big) 
\cr &\le&
 2\sqrt n\, \P\{B_n-B_m\le (n-m)\theta\}\cr 
& & + \sqrt n\, \E_{(V,\e)}{\bf 1}_{\{  B_n-B_m> (n-m)\theta\}} \Big( {1\over (n\theta)^{3/2}}+{1\over (B_{n }-B_{ m})^{3/2}}\Big)
\cr &\le&
 C\, \sqrt n \Big\{ \rho^{  n-m }  +  {1\over (n\theta)^{3/2}}+ {1\over
((n-m)\theta)^{3/2}}\Big\}.
\end{eqnarray}

We now   estimate $A''_1$, which we    bound   as follows:
\begin{eqnarray}\label{e3}A_1''
 &\le& C \sqrt n\ \E_{(V,\e)}{\bf 1}_{\{B_n> n\theta, B_n>B_m\}} \bigg\{ \Big| {
e^{-{(\kappa_n-W_n-(B_n/2))^2\over D^2(B_n/2)}}\over
\sqrt{  B_n}}   
 \cr & &\  -
 { e^{-{(\kappa_n-\kappa_m-(W_{n  }-W_m)-(B_{n }-B_{ m})/2 )^2\over D^2(B_{n }-B_{ m})/2 }}\over
\sqrt{  (B_{n }-B_{ m})} }\Big|\bigg\} 
\cr &\le&  C  \ \E_{(V,\e)}{\bf 1}_{\{B_n> n\theta, B_n>B_m\}} \Big\{\Big({n\over B_n}\Big)^{1/2}\Big[   
  \sqrt{  {B_n} \over
   B_{n }-B_{ m}   }-1 \Big] 
\cr & &\ \times e^{-{(\kappa_n-\kappa_m-(W_{n  }-W_m)-(B_{n }-B_{ m})/2 )^2\over D^2(B_{n }-B_{ m})/2 }}\Big\}
\cr & &\ + C  \ \E_{(V,\e)}{\bf 1}_{\{B_n> n\theta, B_n>B_m\}}\Big\{\Big({n\over B_n}\Big)^{1/2}\Big|   e^{-{(\kappa_n-W_n-(B_n/2))^2\over
D^2(B_n/2)}} 
\cr & &\ -
   e^{-{(\kappa_n-\kappa_m-(W_{n  }-W_m)-(B_{n }-B_{ m})/2 )^2\over D^2(B_{n }-B_{ m})/2 }}\Big|
 \cr &\le & C_\theta \bigg\{  \ \E_{(V,\e)}{\bf 1}_{\{B_n> n\theta, B_n>B_m\}} \Big[   
  \sqrt{  {B_n} \over
   B_{n }-B_{ m}   }-1 \Big]  
 \cr & &\ +   \ \E_{(V,\e)}{\bf 1}_{\{B_n> n\theta, B_n>B_m\}} \Big|   e^{-{(\kappa_n-W_n-(B_n/2))^2\over
D^2(B_n/2)}} 
\cr & &\ -
   e^{-{(\kappa_n-\kappa_m-(W_{n  }-W_m)-(B_{n }-B_{ m})/2 )^2\over D^2(B_{n }-B_{ m})/2 }}\Big| \bigg\}
\cr &:=  &C_\theta \big\{A''_{11}+A''_{12}  \big\}.
\end{eqnarray}

%%%%%%%%%%%%%%%%%%%%%%%%%%%%%%%%%%%%%%%%%%%%%%%%%%%%%%%%%%%%%%%%%%%%%%%%%%%%%%%%%%%%%%%%%%%%%%%%%%%%%%%%%%%%%%%%%%%%%%%%%%%%%%

%%%%%%%%%%%%%%%%%%%%%%%%%%%%%%%%%%%%%%%%%%%%%%%%%%%%%%%%%%%%%%%%%%%%%%%%%%%%%%%%%%%%%%%%%%%%%%%%%%%%%%%%%%%%%%%%%%%%%%%%%%%%%%
%%%%%%%%%%%%%%%%%%%%%%%%%%%%%%%%%%%%%%%%%%%%%%%%%%%%%%%%%%%%%%%%%%%%%%%%%%%%%%%%%%%%%%%%%%%%%%%%%%%%%%%%%%%%%%%%%%%%%%%%%%%%%%

 In the one hand, on the set $\{ B_n>B_m\}$
  \begin{eqnarray*}   \sqrt{  {B_n} \over
   B_{n }-B_{ m}   }-1 &=&  {\sqrt{B_n} -\sqrt{B_{n }-B_{ m}}\over \sqrt{B_{n }-B_{ m}}}= {\sqrt{B_m}\over \sqrt{B_{n }-B_{ m}}}
{\sqrt{B_m}\over \sqrt {B_n}+\sqrt{B_{n }-B_{ m}}}\cr 
   &\le &   {\sqrt{B_m}\over \sqrt{B_{n }-B_{ m}}}
.
   \end{eqnarray*} 
Thus 
\begin{eqnarray}\label{e4} A''_{11}&= &\E_{(V,\e)}{\bf 1}_{\{B_n> n\theta, B_n>B_m\}} \Big[   
  \sqrt{  {B_n} \over
   B_{n }-B_{ m}   }-1 \Big]
\cr  &\le &\E_{(V,\e)}{\bf 1}_{\{B_n> n\theta, B_n>B_m\}}  {\sqrt{B_m}\over \sqrt{B_{n }-B_{ m}}} 
 \cr  &\le &   \E_{(V,\e)}{\bf 1}_{\{ 
B_n>B_m\}}  {\sqrt{B_m}\over \sqrt{B_{n }-B_{ m}}}
\cr  &= &  \E_{(V,\e)}\Big({\bf 1}_{\{  B_n-B_m\le (n-m)\theta\}}
\cr  &   &\ +{\bf 1}_{\{  B_n-B_m>(n-m)\theta\}}\Big) {\bf 1}_{\{ 
B_n>B_m\}}  {\sqrt{B_m}\over
\sqrt{B_{n }-B_{ m}}}   
\cr  &\le & C_\theta \P \{  B_n-B_m\le (n-m)\theta\} +  {C_\theta\over
\sqrt{(n-m)\theta}}\E_{(V,\e)} 
 \sqrt{B_m} 
\cr  &\le & C_\theta \Big\{\rho^{n-m}  +   {\sqrt m\over
\sqrt{(n-m) }} \Big\}=C_\theta \Big\{\rho^{n-m}  +   {1\over
\sqrt{{n\over m}-1 }} \Big\}
\cr  &\le &  C_\theta \Big\{\rho^{n-m}  +   {1\over
\sqrt{{n\over m}}-1 } \Big\},  
\end{eqnarray} 
since $\sqrt x-\sqrt y\le \sqrt{x-y}$ if $x\ge y\ge 0$. 
\medskip\par

   Now  we turn to  $A''_{12}$. Put $\kappa_n'=\kappa_n-W_n-(B_n/2)$. 
Then
\begin{eqnarray*}A''_{12}&=& \E_{(V,\e)}{\bf 1}_{\{B_n> n\theta,  B_n>B_m\}} \Big|   e^{-{(\kappa_n-W_n-(B_n/2))^2\over
D^2(B_n/2)}}
\cr & &  \qq\qq\qq\qq\qq\qq-
   e^{-{(\kappa_n-\kappa_m-(W_{n  }-W_m)-(B_{n }-B_{ m})/2 )^2\over D^2(B_{n }-B_{ m})/2 }}\Big|
\cr &=&\E_{(V,\e)}{\bf 1}_{\{B_n> n\theta,  B_n>B_m\}} \Big|   e^{-{{\kappa'_n}^2\over
D^2(B_n/2)}}
  -
   e^{-{(\kappa'_n-\kappa'_m )^2\over D^2(B_{n }-B_{ m})/2 }}\Big| .
\end{eqnarray*}
We have 
\begin{eqnarray}\label{e5} & &\E_{(V,\e)}{\bf 1}_{\{B_n> n\theta,  0<B_n-B_m\le \theta (n-m)\}} \Big|   e^{-{{\kappa'_n} ^2\over
D^2(B_n/2)}}
  -
   e^{-{(\kappa'_n-\kappa'_m )^2\over D^2(B_{n }-B_{ m})/2 }}\Big| \cr 
&\le& 2\P\{B_n-B_m\le \theta (n-m) \} \le 2\rho^{n-m}.
  \end{eqnarray}
It remains to bound 
$$\E_{(V,\e)}{\bf 1}_{\{B_n> n\theta,   B_n-B_m> \theta (n-m)\}} \Big|   e^{-{{\kappa'_n}^2\over
D^2(B_n/2)}}
  -
   e^{-{(\kappa'_n-\kappa'_m )^2\over D^2(B_{n }-B_{ m})/2 }}\Big|.$$
Let
$$b_n ={ \kappa_n' \over
 \sqrt B_n }     . $$
 By using the inequality $|e^{-u }-e^{-v }|\le
|u-v|$  valid for all reals $u\ge 0,v\ge 0$, we have 
  \begin{eqnarray}    
 & &{   D^2\over 2}\, \Big|   e^{-{{\kappa'_n}^2\over
D^2(B_n/2)}}
  -
   e^{-{(\kappa'_n-\kappa'_m )^2\over D^2(B_{n}-B_{m})/2 }}\Big|
\cr & \le &  \Big| 
-{(\kappa'_n-\kappa'_m )^2\over  (B_{n }-B_{ m})  }
 + {{\kappa'_n}^2\over
 B_n } \Big|= \Big|  -{ (\sqrt{
B_{n}}b_n-
\sqrt{ B_{m}}b_m)^2  \over B_{n}-B_{m} } + b_n^2 \Big|\cr 
 & = &  \Big|   {  -B_{n}b_n^2-B_{m}b_m^2+2\sqrt{B_{n}B_{m}}b_nb_m  \over B_{n}-B_{m} } + b_n^2 \Big|\cr 
 & = &  \Big|   {  -B_{n}b_n^2-B_{m}b_m^2+2\sqrt{B_{n}B_{m}}b_nb_m  +B_{n}b_n^2-B_{m}b_n^2\over B_{n}-B_{m} }  \Big|\cr 
  & = &  \Big|   {   -  b_m^2+2\sqrt{B_{n}\over B_{m}}b_nb_m  - b_n^2\over {B_{n}\over B_{m}}-1 }  \Big|= \Big|   { -(b_n-b_m)^2+
2b_mb_n(\sqrt{B_{n}\over B_{m}}-1)\over {B_{n}\over B_{m}}-1 }  \Big|
  \cr &\le &    { 2(b_n^2+b_m^2)+ 2|b_m||b_n|(\sqrt{B_{n}\over B_{m}}-1)\over {B_{n}\over B_{m}}-1 }     
           . \end{eqnarray} 
 
Hence
\begin{eqnarray}\label{e7} & &\E_{(V,\e)}{\bf 1}_{\{B_n> n\theta,   B_n-B_m> \theta (n-m)\}} \Big|   e^{-{{\kappa'_n}^2\over
D^2(B_n/2)}}
  -
   e^{-{(\kappa'_n-\kappa'_m )^2\over D^2(B_{n }-B_{ m})/2 }}\Big|
\cr & \le &C \E_{(V,\e)}{\bf 1}_{\{B_n> n\theta,   B_n-B_m> \theta (n-m)\}} \bigg\{ {  b_n^2+b_m^2  \over {B_{n}\over B_{m}}-1 }    +{
 |b_m||b_n|  \over \sqrt{B_{n}\over B_{m}}-1  }    \bigg\}. 
 \end{eqnarray}
On the set $ \{B_n> n\theta,   B_n-B_m> \theta (n-m)\} $, we notice that 
$$ {     1\over 
\sqrt{B_{n}\over B_{m}}-1   } = {    \sqrt{B_{m}}\over 
\sqrt{B_{n} }-\sqrt{  B_{m}}   }=  {    \sqrt{B_{m}}(\sqrt{B_{n} }+\sqrt{  B_{m}})\over 
 {B_{n} }- {  B_{m}}   }\le  {    \sqrt{B_{m}}(\sqrt{B_{n} }+\sqrt{  B_{m}})\over 
 \theta(n-m)   }.$$
We also observe that $$\E S_m=\E_{(V,\e)}\E_{\!L}\big(W_m+ D\sum_{j=1}^{ m}\e_jL_j\big)=
\E_{(V,\e)}\big(W_m+ {DB_m\over 2}\big)=m\m .$$
Thus   $W_m+(DB_m/2)-m\m= W_m+(DB_m/2)-\E_{(V,\e)}\big(W_m+ {DB_m / 2}\big)$. 
 
Besides
\begin{eqnarray*} |b_j|&=& {\big|\kappa_j-j\m -\big(W_j+(B_j/2)-j\m\big)\big| \over
 \sqrt  {B_j} } 
\cr &\le & {C\over
 \sqrt  {B_j} }  \Big[ \sqrt j +\big|S'_j-\E_{(V,\e)} S'_j \big| \Big]  , 
\end{eqnarray*}
where we have  denoted $S'_n=W_m+(B_m/2)$. We have
 \begin{eqnarray*}  { |b_n||b_m| \over 
\sqrt{B_{n}\over B_{m}}-1   } &\le&   {C\over
 \sqrt  {B_nB_m} } {   
\sqrt{B_{m}}(\sqrt{B_{n} }+\sqrt{  B_{m}})\over 
 \theta(n-m)   } \cr & &\  \times \Big[ \sqrt n +\big|S'_n-\E_{(V,\e)} S'_n \big| \Big]   \Big[ \sqrt m +\big|S'_m-\E_{(V,\e)} S'_m
\big|
\Big]  
\cr 
&\le&     {   
 C\over 
 \theta(n-m)   } \cr & &\  \times \Big[ \sqrt n +\big|S'_n-\E_{(V,\e)} S'_n \big| \Big]   \Big[ \sqrt m +\big|S'_m-\E_{(V,\e)} S'_m
\big|
\Big]  \cr
&\le&     {   
 C\sqrt{nm}\over 
 \theta(n-m)   } \cr & &\  \times {\big[ \sqrt n +\big|S'_n-\E_{(V,\e)} S'_n \big| \big]    \big[ \sqrt m +\big|S'_m-\E_{(V,\e)} S'_m
\big|
\big]  \over\sqrt{nm} }\cr
 &\le&     {   
 C\sqrt{ m}(\sqrt{n }+\sqrt{ m})\over 
 \theta(n-m)   } \cr & &\  \times  \Big[ 1 +{\big|S'_n-\E_{(V,\e)} S'_n \big| \over \sqrt n}\Big]    \Big[1 +{\big|S'_m-\E_{(V,\e)}
S'_m
\big|
   \over\sqrt{ m} }\Big]\cr    &= &       {   C\over 
 \theta\big(\sqrt{n\over m}-1  \big) } \Big[ 1 +{\big|S'_n-\E_{(V,\e)} S'_n \big| \over \sqrt n}\Big]    \Big[1 +{\big|S'_m-\E_{(V,\e)}
S'_m
\big|
   \over\sqrt{ m} }\Big] .
\end{eqnarray*}
By the Cauchy-Schwarz inequality, 
$$\E_{(V,\e)}  {\big|S'_n-\E_{(V,\e)} S'_n \big| \over \sqrt n}\,  {\big|S'_m-\E_{(V,\e)}
S'_m
\big|
   \over\sqrt{ m} } $$
$$\le \Big[\E_{(V,\e)}  {\big|S'_n-\E_{(V,\e)} S'_n \big|^2 \over   n}\Big]^{1/2}    \Big[ \E_{(V,\e)}{\big|S'_m-\E_{(V,\e)}
S'_m
\big|^2
   \over  m  }\Big]^{1/2}\le C .$$
And also
$$\E_{(V,\e)}  {\big|S'_j-\E_{(V,\e)} S'_j \big| \over \sqrt j} \le \Big[\E_{(V,\e)}  {\big|S'_j-\E_{(V,\e)} S'_j \big|^2 \over  
j}\Big]^{1/2}  \le C .$$
Since
$$\E_{(V,\e)}  { |b_n||b_m| \over 
\sqrt{B_{n}\over B_{m}}-1   } \le     {   C_\theta\over 
 \sqrt{n\over m}-1   } \E_{(V,\e)}\Big[ 1 +{\big|S'_n-\E_{(V,\e)} S'_n \big| \over \sqrt n}\Big]    \Big[1 +{\big|S'_m-\E_{(V,\e)}
S'_m
\big|
   \over\sqrt{ m} }\Big] 
$$
we deduce
\begin{equation}\label{e8} \E_{(V,\e)}\ { |b_n||b_m| \over 
\sqrt{B_{n}\over B_{m}}-1   } \le     {   C_\theta\over 
 \sqrt{n\over m}-1   }  .
\end{equation}

Now  
 \begin{eqnarray*}  { |b_n| ^2\over 
 {B_{n}\over B_{m}}-1   } &\le&  {C\over
   {B_n} }  \Big[ \sqrt n +\big|S'_n-\E_{(V,\e)} S'_n \big| \Big]^2 { B_{m}\over 
  B_{n}    -B_{m}   }
\cr  &\le&   C  { nB_{m}\over 
  {B_n}(B_{n}    -B_{m})   } \Big[ 1 +{\big|S'_n-\E_{(V,\e)} S'_n \big|\over \sqrt n} \Big]^2   
\cr  &\le&   C  {  B_{m}\over 
   \theta^2(n-m)   } \Big[ 1 +{\big|S'_n-\E_{(V,\e)} S'_n \big|\over \sqrt n} \Big]^2 
\cr  &\le&   C  {  m\over 
   \theta^2(n-m)   } \Big[ 1 +{\big|S'_n-\E_{(V,\e)} S'_n \big|\over \sqrt n} \Big]^2 
\cr  & =&      { C\over 
   \theta^2({n\over m}-1)   } \Big[ 1 +{\big|S'_n-\E_{(V,\e)} S'_n \big|\over \sqrt n} \Big]^2
 \cr  &\le&      { C_\theta\over 
    \sqrt {n\over m}-1   } \Big[ 1 +{\big|S'_n-\E_{(V,\e)} S'_n \big|\over \sqrt n} \Big]^2 . 
\end{eqnarray*}
    Therefore
\begin{eqnarray}\label{e9}& &\E_{(V,\e)} {\bf 1}_{\{B_n> n\theta,   B_n-B_m> \theta (n-m)\}} {   |b_n| ^2\over 
 {B_{n}\over B_{m}}-1   }\cr
&\le&   { C_\theta\over 
    \sqrt {n\over m}-1   } \E_{(V,\e)}   \Big[ 1 +{\big|S'_n-\E_{(V,\e)} S'_n \big|\over \sqrt n} \Big]^2
\cr
&\le&   { C_\theta\over 
    \sqrt {n\over m}-1   }    \Big[ 1 +{1\over   n}\E_{(V,\e)}  \big|S'_n-\E_{(V,\e)} S'_n \big|^2  \Big] \cr
&\le&   { C_\theta\over 
    \sqrt {n\over m}-1   } .  
\end{eqnarray}

Next $|b_m|  \le   {C\over
 \sqrt  {B_m} }  \Big[ \sqrt m +\big|S'_m-\E_{(V,\e)} S'_m \big| \Big]  . 
 $
 \begin{eqnarray*}  { |b_m| ^2\over 
 {B_{n}\over B_{m}}-1   } &\le&  {C\over
   {B_m} }  \Big[ \sqrt m +\big|S'_m-\E_{(V,\e)} S'_m \big| \Big]^2 { B_{m}\over 
  B_{n}    -B_{m}   }
\cr  &\le&   C  { mB_{m}\over 
  {B_m}(B_{n}    -B_{m})   } \Big[ 1 +{\big|S'_m-\E_{(V,\e)} S'_m \big|\over \sqrt m} \Big]^2   
\cr  &\le&   C  {  m\over 
   \theta (n-m)   } \Big[ 1 +{\big|S'_m-\E_{(V,\e)} S'_m \big|\over \sqrt m}\Big]^2 
 \cr  & =&      { C\over 
   \theta ({n\over m}-1)   } \Big[ 1 +{\big|S'_m-\E_{(V,\e)} S'_m \big|\over \sqrt m} \Big]^2
 \cr  &\le&      { C_\theta\over 
    \sqrt {n\over m}-1   } \Big[ 1 +{\big|S'_m-\E_{(V,\e)} S'_m \big|\over \sqrt m} \Big]^2 . 
\end{eqnarray*}
    Therefore
\begin{eqnarray}\label{e10}& &\E_{(V,\e)} {\bf 1}_{\{B_n> n\theta,   B_n-B_m> \theta (n-m)\}} {   |b_m| ^2\over 
 {B_{n}\over B_{m}}-1   }\cr
&\le&   { C_\theta\over 
    \sqrt {n\over m}-1   } \E_{(V,\e)}   \Big[ 1 +{\big|S'_m-\E_{(V,\e)} S'_m \big|\over \sqrt m} \Big]^2
\cr
&\le&   { C_\theta\over 
    \sqrt {n\over m}-1   }    \Big[ 1 +{1\over   m}\E_{(V,\e)}  \big|S'_m-\E_{(V,\e)} S'_m \big|^2  \Big] \cr
&\le&   { C_\theta\over 
    \sqrt {n\over m}-1   } .  
\end{eqnarray}

By inserting estimates  (\ref{e8}),(\ref{e9}),(\ref{e10}) into  (\ref{e7}), we get
\begin{eqnarray}\label{e11} & &\E_{(V,\e)}{\bf 1}_{\{B_n> n\theta,   B_n-B_m> \theta (n-m)\}} \Big|   e^{-{{\kappa'_n}^2\over
D^2(B_n/2)}}
  -
   e^{-{(\kappa'_n-\kappa'_m )^2\over D^2(B_{n }-B_{ m})/2 }}\Big|
\cr & \le &{ C_\theta\over 
    \sqrt {n\over m}-1   }   .
\end{eqnarray} 
 This estimate along with (\ref{e5}) yields, in view of (\ref{e3}),
\begin{eqnarray}\label{e12} A''_{1} &\le & 2\rho^{n-m}+ { C_\theta\over 
    \sqrt {n\over m}-1   } 
   .
\end{eqnarray} 
And  with (\ref{e2}), (\ref{e6}), (\ref{e3}),  (\ref{e4})
\begin{eqnarray}\label{e13} |A''| &\le & C_{\theta  } \Big\{\rho^{n-m}  
 +   {   1\over 
 \sqrt{n\over m}-1   }+ \sqrt n \Big\{{1 \over
 n ^{3/2}}+   \rho^{  n-m }     +  {1 \over
 (n-m) ^{3/2}}\Big\}\Big\}
   .\cr & & 
\end{eqnarray}
Consequently, with   (\ref{e1})
\begin{eqnarray}\label{e14} |A' |+|A''| &\le & C_{\theta  } \Big\{\rho^{n-m}  
 +   {   1\over 
 \sqrt{n\over m}-1   } + \sqrt n \Big\{{1 \over
 n ^{3/2}}+  \rho^{  n-m }  \cr & &   +2^{-(n-m)} +  {1 \over
 (n-m) ^{3/2}}\Big\}\Big\}
   .
\end{eqnarray}
 Finally with  (\ref{e0}),  
\begin{eqnarray}\label{e15}  |A |  &\le &  C_{\theta  } \Big\{\rho^{n-m}  
 +   {   1\over 
 \sqrt{n\over m}-1   } + \sqrt n \Big\{{1 \over
 n ^{3/2}}+  \rho^{  n-m }  \cr & &   +2^{-(n-m)} +  {1 \over
 (n-m) ^{3/2}}\Big\}\Big\}
   .
\end{eqnarray}
And with (\ref{cov}), we obtain 
  \begin{eqnarray}\label{e14} \big|\E   Y_nY_m\big|  &\le &  C_{\theta  } \Big\{\rho^{n-m}  
 +   {   1\over 
 \sqrt{n\over m}-1   } + \sqrt n \Big\{{1 \over
 n ^{3/2}}+  \rho^{  n-m }  \cr & &   +2^{-(n-m)} +  {1 \over
 (n-m) ^{3/2}}\Big\}\Big\}
\cr 
 &\le &  C_{\theta  } \Big\{     {   1\over 
 \sqrt{n\over m}-1   } +      {n^{1/2}  \over
 (n-m) ^{3/2}} \Big\}
   .
\end{eqnarray}
This proves (\ref{C1}).
Now let $0<c<1$. Let $m\le cn$. Then
$$   {1\over    \sqrt{ n/m}-1   }\le \Big({1\over 1-\sqrt c} \Big) \, \sqrt{m\over n}. $$
 Further
  $$     {\sqrt n   \over
 (n-m) ^{3/2}}=  \sqrt {n\over n-m}\     {  1\over
 (n-m)  }  =  {1\over{ (1- m/n)^{3/2} }}
\ {1 \over n}\le {1\over{ (1- c)^{3/2} }}
\ {1  \over n  }  . $$
By incorporating these estimates into (\ref{e10}) we get 
\begin{eqnarray}\label{e11} \big|\E  Y_nY_m\big|  &\le & C_{\theta  }          
\,
\sqrt{m\over n}  .
\end{eqnarray} 
This establishes (\ref{C2}). The proof is now complete. 
\begin{remark} Although the rate of approximation in the  local limit theorem (\ref{lltber}) for Bernoulli sums used in the proof, be quite sharp, it seems worth to indicate 
 that a better   rate can be obtained,  with a different centering term however.   More precisely, 
%  $$\sup_{j}\Big|\P\{R_n= j\} - {(1+(-1)^{n+j})\over 2\pi}  \int_\R  e^{ijv -  n ({v^2 \over   2}+{ v^4 \over 12} ) }\, d v
%\Big|\le C  {
%\log^{7/2} n 
%\over n^{ 5/2}}.  $$  
$$\sup_{j}\Big|\P\{B_n= k\} - { 1\over  \pi}  \int_\R  e^{i(2k-n)v -  n ({v^2 \over   2}+{ v^4 \over 12} ) }\, d v\Big|\le C  {
\log^{7/2} n 
\over n^{ 5/2}}.  $$  
  The constant $C$ is  absolute.   As   $ \int_\R  e^{ijv -  n ({v^2 \over   2}+{ v^4 \over 12} ) }\, d v= {1\over \sqrt n}\int_\R 
e^{i{j\over \sqrt n} w -     {w^2
\over   2} - { w^4 \over 12n}  }\, d w $, 
($j=2k-n$),     a corrective factor $ e^{  - { w^4 \over 12n}  }$  depending on $n$ appears in this   formulation. 

\end{remark}

%%%%%%%%%%%%%%%%%%%%%%%%%%%%%%%%%%%%%%%%%%%%%%%%%%%%%%%%%%%%%%%%%%%%%%%%%%%%%%%%%%%%%%%%%%%%%%%%%%%%%%%%%%%%%%%%%%%%%%%%%%%%%%
%%%%%%%%%%%%%%%%%%%%%%%%%%%%%%%%%%%%%%%%%%%%%%%%%%%%%%%%%%%%%%%%%%%%%%%%%%%%%%%%%%%%%%%%%%%%%%%%%%%%%%%%%%%%%%%%%%%%%%%%%%%%%%
%%%%%%%%%%%%%%%%%%%%%%%%%%%%%%%%%%%%%%%%%%%%%%%%%%%%%%%%%%%%%%%%%%%%%%%%%%%%%%%%%%%%%%%%%%%%%%%%%%%%%%%%%%%%%%%%%%%%%%%%%%%%%%
%%%%%%%%%%%%%%%%%%%%%%%%%%%%%%%%%%%%%%%%%%%%%%%%%%%%%%%%%%%%%%%%%%%%%%%%%%%%%%%%%%%%%%%%%%%%%%%%%%%%%%%%%%%%%%%%%%%%%%%%%%%%%%
\section{Application} 
In this section, we deduce   from Theorem \ref{t1} an almost sure local limit theorem for iid   square integrable random variables taking
values in an arbitrary lattice $\mathcal L( v_0,D)$. In
\cite{DK} (sections 1,2), the notion of almost sure local limit theorem is introduced in analogy with the usual almost sure central limit
theorem:
 "A stationary sequence of random variables $\{X_n, n\ge 1\}$ taking values in $\R$ or $\Z$ with partial sums $S_n =X_1= \ldots +X_n$ satisfies an almost
sure local limit theorem, if there exist sequences
$\{a_n, n\ge 1\}$ in
$\R$ and
$\{b_n, n\ge 1\}$ in $\R^+$ satisfying $b_n\to \infty$, such that 
\begin{equation}\label{asllt}\lim_{N\to \infty}{1\over \log N}\sum_{n=1}^N {b_n\over n} \chi\{ S_n\in k_n+I\}\buildrel{a.s.}\over{ =}g(\kappa)|I| \quad {\rm
as}\quad {k_n-a_n\over b_n}\to
\kappa, 
\end{equation}  where $g$ denotes some density and $I\subset \R$ is some bounded interval.  Further $|I|$ denotes the length of the interval $I$
in the case where $X_1$ is real valued and the counting measure of $I$ otherwise." 
\smallskip \par  In what follows, we restrict our consideration to the iid case. We assume
that 
  $\P\{X_1\in \mathcal L(v_0,D)\}=1$, $\mathcal L(v_0,D)\subset \Z$.
 We also assume that $\s^2= \E X_1^2<\infty$ and let  
$\m=\E X_1$. 
\smallskip \par 
We begin with an elementary observation. Let $v_0=0$ for simplicity. As $g$ is a density, there are reals $\kappa$ such that
$g(\kappa)\not= 0$.  Clearly, if
$\{ k_n, n\ge 1\}$ is such that  ${k_n-a_n\over
b_n}\to
\kappa$, then   any  sequence $\{\k_n, n\ge 1\}$, $\k_n=k_n +u_n$ where $u_n$ are uniformly bounded  also satisfies this. But we can
arrange the $u_n$ so that $\k_n\notin \mathcal L( 0,D)$ for all $n$. Therefore $\P\{S_n=\k_n\}\equiv 0$. If $I=[-\d, \d]$ with $ \d<1/2$,
then
$|I|=1$ and we see that, {\it no matter} the sequences $\{a_n, n\ge 1\}$ and
$\{b_n, n\ge 1\}$ are, property (\ref{asllt}) cannot hold for the sequence  $\{\k_n, n\ge 1\}$, since   
$$ \lim_{N\to \infty}{1\over \log N}\sum_{n=1}^N {b_n\over n} \chi\{ S_n\in k_n+I\}\buildrel{a.s.}\over{ =}0\not=g(\kappa)|I| .
$$\smallskip\par 
It thus appears   necessary (also when $v_0$ is arbitrary) to complete the above definition by introducing  the additional requirement: 
\begin{equation}\label{addasllt} \k_n \in \mathcal L(nv_0,D), \qq n= 1,2, \ldots. 
\end{equation}
Then   $k_n+I\subset L(nv_0,D)$ if and only if $ I\subset L(0,D)$. And also to change $|I|$ for $\#\{ I\cap
\mathcal L( 0,D)\}$. Then (\ref{asllt}) is modified   as follows:
\begin{equation}\label{asllt1}\lim_{N\to \infty}{1\over \log N}\sum_{n=1}^N {b_n\over n} \chi\{ S_n\in k_n+I\}\buildrel{a.s.}\over{
=}g(\kappa)\#\{ I\cap \mathcal  L(  0,D)\}, \  {\rm as}\  {k_n-a_n\over b_n}\to
\kappa,  
\end{equation}
where $I$ is a bounded interval. This is    coherent with the   local limit theorem 
which   relies upon the three parameters $\m$, $\s$ and the
 (maximal) span of $X_1$.   It is obvious
by invoking a simple additivity argument,  that (\ref{asllt1}) holds for any bounded interval $I$ if and only if 
\begin{equation}\label{asllt2}\lim_{N\to \infty}{1\over \log N}\sum_{n=1}^N {b_n\over n} \chi\{ S_n= k_n \}\buildrel{a.s.}\over{
=}g(\kappa) , \  {\rm as}\  {k_n-a_n\over b_n}\to
\kappa.  
\end{equation}
% If $v_0\not=0$, the same remarks are in order up to (\ref{addasllt}) in which the condition becomes $\k_n \in \mathcal L(nv_0,D)$
%for all $n$. But after, we don't know how to work with any given bounded interval $I$, the requirement   being that $\k_n +I\in
%\mathcal L(nv_0,D)$ for all $n$. Therefore, we will consider that $X$ satisfies an almost sure local limit theorem when (\ref{asllt2}) is
%realized, and further $ \k_n \in \mathcal L(nv_0,D)$ for all $n$.
\smallskip \par  As mentionned by the authors  in \cite{DK}, p.146, the existence of almost sure local limit theorems is of fundamanental
interest.  A recent application to a problem of representation of integers in given in \cite{W1}. By (\ref{llt}),  the
 local limit theorem holds,
 and if
$ \k_n \in \mathcal L(nv_0,D)$ is a sequence which   verifies condition (\ref{2}), namely  $\lim_{n\to \infty} { \k_n-n\m  
\over   
\sqrt{  n}  }=
\k$, then
 \begin{equation}\label{llt1}  \lim_{n\to \infty}  \sqrt n \P\{S_n=\k_n\}={D\over  \sqrt{ 2\pi}\s}e^{- 
{ \k ^2\over  2   \s^2} } . 
\end{equation}

 We deduce from Theorem
\ref{t1}  an almost sure   local limit theorem   for i.i.d. square integrable random variables taking
values in an arbitrary lattice $\mathcal L( v_0,D)$.   
     \begin{theorem} \label{t2} Let $X$ be a square integrable lattice distributed random variable with maximal span $D$. Let $\m =\E X$,
$\s^2=\E X^2-(\E X)^2$. Let also 
$ \{X_k, k\ge 1\}$ be independent copies of
$X$, and put
$S_n=X_1+\ldots +X_n$, $n\ge 1$.   Then    
$$ \lim_{ N\to \infty}{1\over    \log N } \sum_{ n\le
N}  {  1 \over \sqrt n} {\bf 1}_{\{S_n=\kappa_n\}} \buildrel{a.s.}\over {=}{D\over 
\sqrt{ 2\pi}\s}e^{-  {\k^2/ ( 2\s^2 ) } },$$
  for any  sequence of integers $\{\k_n, n\ge 1\}$     such that (\ref{2}) holds. 
 \end{theorem}
\begin{remark}\rm In \cite{DK},  Corollary 2 (see also p.148-149) the authors show that "the almost sure local limit theorem holds for iid
sequences of square integrable $\Z$-valued random variables, that is:
$$ \lim_{ N\to \infty}{1\over    \log N } \sum_{ n\le
N}  {  1 \over \sqrt n} {\bf 1}_{\{S_n=\kappa_n\}} \buildrel{a.s.}\over {=}{1\over 
\sqrt{ 2\pi}\s}e^{-  {\k^2/ ( 2\s^2 ) } }\quad {\rm if}\ {k_n-n\m\over \sqrt n}\to
 \kappa."$$
By the remarks made before concerning (\ref{asllt}), this statement needs a correction.
The proof is sketched as follows. Let
$\phi$ denote the characteristic function of $X$. By Fourier inversion formula
 $\P\{S_n=k\}=\int_0^1e^{-2i\pi kt}\phi^n(t) dt$.   
 Thus 
$$\sqrt {nm}\P\{ S_n=k_n, S_m=k_m\}=\sqrt {nm}\P\{  S_m=k_m\}\P\{ S_n-  S_m=k_n-k_m\}$$ $$
=\sqrt {nm}\int_0^1e^{-2i\pi k_mt}\phi^m(t) dt\int_0^1e^{-2i\pi (k_n-k_m)t}\phi^{n-m}(t) dt$$  
$$
=\sqrt {n\over n-m }\int_0^{\sqrt m}e^{-2i\pi {k_m\over \sqrt m}u}\phi^m({u\over \sqrt m}) du\int_0^{\sqrt{n-m}}e^{-2i\pi ({k_n-k_m\over
\sqrt{n-m}})u}\phi^{n-m}({v\over \sqrt{n-m}}) dv.$$
By the CLT, 
 $$ \lim_{m\to \infty}\phi^m({u\over \sqrt m})= {e^{-\k^2/2}\over \sqrt{2\pi}} , \qq \lim_{n-m\to \infty} \phi^{n-m}({v\over \sqrt{n-m}}) =
 {e^{-\k^2/2}\over \sqrt{2\pi}} .$$
   Next it is claimed that it  implies
$$"\sqrt {nm}\P\{ S_n=k_n, S_m=k_m\} \to \sqrt {n\over n-m }{1\over 2\pi}e^{-\k^2}." $$
We presume that this should rather be $\sqrt {nm}\P\{ S_n=k_n, S_m=k_m\} \to  {1\over 2\pi}e^{-\k^2}$. 
However, we have   not been able to check this. From our main result,
we only get
$$\lim_{n,m\to \infty\atop {n\over m}\to \infty} \sqrt{nm} \, \Big|\P\{S_n=\k_n, S_m=\k_m\}-\P\{S_n=\k_n \}\P\{  S_m=\k_m\} \Big|  
=0   .
$$
The authors argue that the proof maybe be continued as in the Bernoulli case where a theorem of Mori is invoked. This one requires to
have at disposal a correlation bound. For having tried to apply Mori's result with our correlation inequality in Theorem \ref{t1},
this only allowed us to treat   subsequences
$n=n_k$ with
$n_{k=1}/n_k\ge
\sqrt 2$.  We  believe that the proof needs some complementary explanations.\end{remark} The
notion of     quasi-orthogonal system  is used in the proof of Theorem \ref{t2}. We recall it briefly.  
    A sequence $  \underline f=\{f_n, n\ge 1\}$   in an Hilbert space $H$ is called (see \cite{KSZ} or \cite{W2} p.22)   a
quasi-orthogonal system  if the quadratic form   on $\ell_2$
defined by
 $\{x_h, h\ge 1\}\mapsto \|\sum_{h }  x_h   f_h\|^2 $  is
bounded.   
  A necessary and sufficient condition for $  \underline f$ to be
quasi-orthogonal  is that the series   $\sum c_nf_n$ converges in
$H$, for any sequence $\{c_n, n\ge 1\} $ such that  $\sum c_n^2
<\infty$. This follows from the fact that   $\underline f$ is quasi-orthogonal if and only if there
exists a constant $L$ depending on $\underline f$ only, such that
   $$ \Big\|\sum_{i\le n}
x_if_i\Big\| \le L \Big(\sum_{i\le n}|x_i|^2 \Big)^{1/2}.$$ 
     Further, as observed in  \cite{KSZ}: "{\it Every theorem on
orthogonal systems  whose proof depends only on Bessel's inequality,
holds for quasi-orthogonal systems}".
   In particular for $H= L^2(  X,\A, \mu)$,  $(  X,\A, \mu)$ a probability space,
  Rademacher--Menchov's theorem applies. We recall it (see \cite{W2} p.363 for instance).
\begin{lemma} Let $\{f_n, n\ge 1\}\subset H$  be an orthogonal sequence.
  The series $\sum c_nf_n $ converges almost everywhere  provided
that $\sum c_n^2 \log^2 n<\infty$.
\end{lemma}  

\begin{proof} We first give the proof     under the additional assumption (\ref{basber}). Next we establish the result without this
one. Assume thus, at first,  that assumption (\ref{basber}) is fulfilled; the  proof is then   identical to the one of Theorem 1 in
\cite{GW3}. 
  Put for any positive integer
$j$
$$ Z_j=\sum_{2^j\le n<2^{j+1}}{Y_n\over n} . $$
By   (\ref{llt}),
 $$ \E Y_n^2=n
\P\{S_n=\kappa_n\}\big(1-\P\{S_n=\kappa_n\} \big) ={\mathcal O}(\sqrt n). 
$$
This and the   second inequality of Theorem \ref{t1}  imply that    $\{Z_j, j\ge 1\}$ is a quasi-orthogonal system. As 
Rademacher-Menchov Theorem applies to quasi-orthogonal systems, the   series 
$$\sum_j {Z_j\over   j^{1/2} (\log j)^{ b} } 
$$ thus converges   almost surely if $b>3/2$. By Kronecker's Lemma
$${1\over   N^{1/2} (\log N)^{ b} }\sum_{j=1}^N Z_j  ={1\over   N^{1/2} (\log N)^{ b} } \sum_{  1\le n<2^{N+1} }{Y_n\over n}\ \to 0, 
$$ as $N$ tends to infinity, almost surely. It is then a routine calculation to derive from this that
 $$\lim_{  t\to \infty} {  1 \over \log t}\,
\sum_{ n\le t}{  1 \over
\sqrt n} {\bf 1}_{\{S_n=\kappa_n\}} \buildrel{a.s.}\over
{=} g(\k)  .
$$
 \smallskip\par
Now we pass to the general case. On the basis of Remark \ref{llt}, we may "change" $X$ for $X'=S_{n_\k}$. But after  this is not so simple as it looks, 
and some extra work is necessary, in order to make this step precise. 
Let $X'_1,X'_2 \ldots,$ be independent copies of $X'$, which we assume to be
also independent from the sequence $X_1,X_2 \ldots,$  and denote similarly
$S'_m=X'_1+\ldots+X'_m$, $m\ge 1$.  
The first point to observe is that, given
$0\le a<n_\k$, the sequence 
$\{S_{a+ mn_\k},m \ge 1\}$ has same law as the sequence $\{S_{a }+S'_m,m \ge 1\}$. It is indeed immediate if we write that  $S_{a+ mn_\k}=S_a+ (S_{a+ 
n_\k}-S_{a })+\ldots (S_{a+ mn_\k}-S_{a+ (m-1)n_\k})$, (and not $S_{a+ mn_\k}= S_{n_\k } +\ldots (S_{  mn_\k}-S_{  (m-1)n_\k})+ 
  (S_{ a+ mn_\k}-S_{ mn_\k})$!). Like this, we are thus preliminary led to consider the sequence $\{S_{a }+S'_m,m \ge 1\}$. 
But this one is a bit outside
from our framework and we have to understand more the role played by  
  the additional independent term $S_a$. 
 Let    $ \k_n \in L(nv_0,D)$, $n=1,2,\ldots$  be a  sequence of integers   such that 
$$ \lim_{n\to \infty}   { \k_n-  n\m   \over    \sqrt{  n}  } = \k . $$ 
Then for any $0\le a<n_\k$,
$$ \lim_{m\to \infty}   { \k_{a+ mn_\k}-  (a+ mn_\k)\m   \over    \sqrt{  m}  } = \k\sqrt{n_\k}. $$ 
  Further,  not only   $ \lim_{m\to \infty}   {   \k_{a+ mn_\k}-    mn_\k \m   \over    \sqrt{ 
m}  } =
\k\sqrt{n_\k}$, but also for any $0\le a<n_\k$,
$$ \lim_{m\to \infty}   {       \k_{a+ mn_\k} -S_a -    mn_\k \m   \over    \sqrt{  m}  } \buildrel{\rm
a.s.}\over {=} \k \sqrt{n_\k}.$$
  By    noticing that
$\{S'_m, m\ge 1\}$ has same law as $\{S_{mn_\k}, m\ge
1\}$, next applying (\ref{llt}) to $X$    and specifying it for the subsequence $\{mn_\k, m\ge 1\} $, we get  
 \begin{equation} \lim_{m\to \infty}\sup_{N=v_0mn_\k+Dk } \Big|  \sqrt m \P\{S'_m=N\}-{D\over \s  \sqrt{ 2\pi n_\k} }e^{- 
{(N-mn_\k\m )^2\over  2 mn_\k \s^2} }\Big|= 0 . 
 \end{equation}
 Let $\tilde
\k_m=  
\k_{a+ mn_\k} -S_a$,
$0\le a<n_\k$ being fixed. We have $\k_{a+ mn_\k}\in L((a+ mn_\k)v_0,D)$,  $S_a\in L( a  v_0,D)$, so $\tilde\k_m \in L(mn_\k v_0,D)$.  Thus,
\begin{eqnarray*} \lim_{m\to \infty}   \sqrt{ m n_\k}\P\{S'_m=\tilde \k_m\} &\buildrel{a.s.}\over {=}&\lim_{m\to \infty} {D\over \s  \sqrt{ 2\pi  }
}e^{-  {(\tilde \k_m-mn_\k\m )^2\over  2 m  \s^2} }  \buildrel{a.s.}\over {=} {D\over \s\sqrt{ 2\pi  } }e^{-  {\k^2/  ( 2 \s^2 ) } }.
\end{eqnarray*}
    Instead of considering $Y_m= \sqrt m \big({\bf
1}_{\{S'_m=\kappa_m\}}-\P\{S'_m=\kappa_m\}
\big)$, we rather work with
$$Y'_m=     \sqrt{ m n_\k}
\big({\bf 1}_{\{S'_m=\kappa_m\}}-\P\{S'_m=\kappa_m\} \big)  .
$$ 
 This   amounts to the same, apart from  the constant factor $\sqrt{n_\k}$. 
By the first step,  
\begin{eqnarray}\label{4}\lim_{N\to \infty}{1\over \log N}\sum_{m=1}^N {Y'_m\over m}    & \buildrel{\rm
a.s.}\over {=}&0 . 
\end{eqnarray}
Thus
 $$ \lim_{N\to \infty}   {1\over \log N}\sum_{m=1}^N  \Big({  \sqrt{ m n_\k}\over
  m}{\bf 1}_{\{S'_m=\kappa_m-S_a\}}  -  {  \sqrt{ m n_\k}\P\{S'_m=\kappa_m-S_a\}\over   m}  \Big)\buildrel{\rm a.s.}\over {=}0.$$
But
$$ \lim_{N\to \infty}   {1\over \log N}\sum_{m=1}^N    {  \sqrt{ m n_\k}\P\{S'_m=\kappa_m-S_a\}\over   m} \buildrel{\rm a.s.}\over {=}{D\over \s\sqrt{
2\pi  } }e^{-  {\k^2/  ( 2 \s^2 ) } }.$$
Therefore
  $$  \lim_{N\to
\infty}   {1\over
\log N}\sum_{m=1}^N { \sqrt{ m n_\k}\over   m}{\bf 1}_{\{ S_a+S'_m=\kappa_m\}}\buildrel{\rm a.s.}\over {=}{D\over \s\sqrt{
2\pi  } }e^{-  {\k^2/  ( 2 \s^2 ) } }.$$  
   We deduce  that
$$ \lim_{ N\to \infty}{1\over    \log N } \sum_{ m\le 
N}   { \sqrt{ m n_\k}\over   m}{\bf 1}_{\{S_{a+ mn_\k}=\k_{a+ mn_\k}\}} \buildrel{\rm a.s.}\over {=}{D\over \s\sqrt{
2\pi  } }e^{-  {\k^2/  ( 2 \s^2 ) } }.$$
Now divide both sides by $n_\k$. We get 
$$ \lim_{ N\to \infty}{1\over    \log N } \sum_{ m\le 
N}   { 1\over  \sqrt{ m n_\k}}{\bf 1}_{\{S_{a+ mn_\k}=\k_{a+ mn_\k}\}} \buildrel{\rm a.s.}\over {=}{D\over n_\k\s\sqrt{
2\pi  } }e^{-  {\k^2/  ( 2 \s^2 ) } }.$$
 But this in turn also implies 
$$ \lim_{ N\to \infty}{1\over    \log N } \sum_{ m\le
N}   { 1\over   \sqrt{ a+ mn_\k}}{\bf 1}_{\{S_{a+ mn_\k}=\k_{a+ mn_\k}\}} \buildrel{\rm a.s.}\over {=}{D\over n_\k\s\sqrt{
2\pi  } }e^{-  {\k^2/  ( 2 \s^2 ) } }.$$
Now $\log N\sim \log (N+1)n_\k$, and by summing the latter over $0\le a<n_\k$, we get  
$$ \lim_{ N\to \infty}{1\over    \log (N+1)n_\k }\sum_{0\le a<n_\k}   \sum_{ m\le
N}   { 1\over   \sqrt{ a+ mn_\k}}{\bf 1}_{\{S_{a+ mn_\k}=\kappa_m\}}
$$
$$= \lim_{ N\to \infty}{1\over    \log (N+1)n_\k }   \sum_{ n\le
(N+1)n_\k}   {1\over   \sqrt{n}}{\bf 1}_{\{S_{n}=\kappa_n\}} \buildrel{\rm a.s.}\over {=}{D\over  \s\sqrt{
2\pi  } }e^{-  {\k^2/  ( 2 \s^2 ) } }.$$
 Hence 
  Theorem \ref{t2} is proved. 
 
\end{proof}

%\medskip\noi{\sl Acknowledgment:} I thank Mikhail Lifshits for pointing me Remark \ref{bbl}. 

%%%%%%%%%%%%%%%%%%%%%%%%%%%%%%%%%%%%%%%%%%%%%%%%%%%%%%%%%%%%%%%%%%%%%%%%%%%%%%%%%%%%%%%%%%%%%%%%%%%%%%%%%%%%%%%%%%%%%%%%%%%%%%
%%%%%%%%%%%%%%%%%%%%%%%%%%%%%%%%%%%%%%%%%%%%%%%%%%%%%%%%%%%%%%%%%%%%%%%%%%%%%%%%%%%%%%%%%%%%%%%%%%%%%%%%%%%%%%%%%%%%%%%%%%%%%%
%%%%%%%%%%%%%%%%%%%%%%%%%%%%%%%%%%%%%%%%%%%%%%%%%%%%%%%%%%%%%%%%%%%%%%%%%%%%%%%%%%%%%%%%%%%%%%%%%%%%%%%%%%%%%%%%%%%%%%%%%%%%%%
%%%%%%%%%%%%%%%%%%%%%%%%%%%%%%%%%%%%%%%%%%%%%%%%%%%%%%%%%%%%%%%%%%%%%%%%%%%%%%%%%%%%%%%%%%%%%%%%%%%%%%%%%%%%%%%%%%%%%%%%%%%%%%
{

\end{document}